\documentclass[12pt]{amsart}
\usepackage{amsfonts}
\usepackage{amssymb}
\numberwithin{equation}{section}
\theoremstyle{plain}
 \newtheorem{thm}{Theorem}[section]
 \newtheorem{lem}[thm]{Lemma}
 
 \newtheorem{prop}[thm]{Proposition}
\theoremstyle{definition}
\newtheorem{defn}[thm]{Definition}
 \newtheorem{ex}[thm]{Example}
 \newtheorem{rem}[thm]{Remark}

\newcommand{\al}{\alpha}
\newcommand{\bt}{\beta}
\newcommand{\gm}{\gamma}
\newcommand{\Gm}{\Gamma}
\newcommand{\ld}{\lambda}
\newcommand{\ta}{\tau}
\newcommand{\Ph}{\Phi}
\newcommand{\Ps}{\Psi}
\newcommand{\rh}{\rho}
\newcommand{\Up}{\Upsilon}
\newcommand{\mcal}{\mathcal}
\newcommand{\mrm}{\mathrm}
\newcommand{\mfr}{\mathfrak}
\newcommand{\wh}{\widehat}
\newcommand{\wt}{\widetilde}
\newcommand{\R}{\mathbb{R}}
\newcommand{\rd}{\mathbb R ^d}
\newcommand{\law}{\mathcal L}
\newcommand{\sek}{\int_0^{\infty}}

\newcommand{\xsm}{X_s^{(\mu)}}
\newcommand{\n}{\noindent}
\newcommand{\les}{\leqslant}
\newcommand{\ges}{\geqslant}
\newcommand{\tr}{\mathrm{tr}\,}

\begin{document}
\setlength{\baselineskip}{18pt}
\setlength{\parindent}{1.8pc}

\title{The limits of nested subclasses of several classes of
infinitely divisible distributions are identical
with the closure of the class of stable distributions}
\maketitle

\n
{\bf Makoto Maejima}$^1$\\  
{\small Department of Mathematics,
Keio University, 3-14-1, Hiyoshi, Kohoku-ku, Yokohama 223-8522, Japan.\\
E-mail: maejima@math.keio.ac.jp; Fax: +81-45-566-1462}\\
{\bf Ken-iti Sato}$^2$\\
{\small Hachiman-yama 1101-5-103, Tenpaku-ku, Nagoya 468-0074, Japan.\\
E-mail: ken-iti.sato@nifty.ne.jp}

\overfullrule=0pt

\vskip 10mm
\n
({\it Running head} : Nested subclasses of
infinitely divisible distributions)

\vskip 10mm
\n
{\bf Abstract}. 
It is shown that the limits of the nested subclasses of five classes of 
infinitely divisible distributions on $\mathbb{R}^d$,
which are the Jurek class, the Goldie--Steutel--Bondesson class, the class 
of selfdecomposable 
distributions, the Thorin class and the class of generalized type $G$ distributions,
are identical with the closure of the class of stable distributions.
More general results are also given.

\vskip 5mm
\n
\vskip 10mm
\section{Introduction}

Subdivision of the class of infinitely divisible distributions 
on $\R^d$ has been an important subject since Urbanik's papers ([16], [17]).
 [7], [6], and [3] are some of many papers in this field. Among others, 
there are the Jurek class, 
the Goldie--Steutel--Bondesson class, the class of selfdecomposable 
distributions, the Thorin class, the class of type $G$ 
distributions and their respective nested subclasses. Jurek ([6]) 
showed that two limits of the nested subclasses starting from the Jurek class and 
the class 
of selfdecomposable distributions are identical.
It is also known (see [17] and [10]) that
the latter is the closure of the class of stable distributions, where 
the closure is taken under weak convergence and convolution.

In this paper, we treat five classes of infinitely divisible distributions on 
$\rd$, all of which are characterized in
terms of the radial components in the polar decomposition of the L\'evy measures of
infinitely divisible distributions, and the purpose of this paper is to show that 
the limits of the nested subclasses of these 
five classes are identical and equal to the closure of the class of stable distributions.
In the course of the proof, we also give a more general theorem.

\vskip 10mm

\section{Preliminaries and the main result}

Throughout the paper, $\law (X)$ denotes the law of an $\rd$-valued random variable 
$X$ and 
$\wh \mu (z), z\in\rd$, denotes the characteristic function of a probability 
distribution $\mu$ on $\rd$.
Also $I(\rd)$ denotes the class of all infinitely divisible distributions on $\rd$,
$I_{{\rm sym}}(\rd) =\{ \mu\in I(\rd) :\mu\,\,\text{is symmetric on}\,\, \rd\} $,
$I_{\log}(\rd)= \{ \mu\in I(\rd) :\int_{\rd}\log ^+|x|\mu(dx)<\infty\} $ and
$I_{\log^{m}}(\rd)= \{ \mu\in I(\rd) :\int_{\rd}(\log ^+|x|)^{m}\mu(dx)<\infty\}$.
Further $S_{\al}(\rd)$ denotes the class of $\al$-stable distributions on $\rd$
for $0<\al\les2$ and $S(\rd)$ denotes the class of stable distributions on $\rd$.
Let $C_{\mu}(z), z\in \rd$, be the cumulant function of $\mu\in I(\rd)$.  
That is, $C_{\mu}(z)$ is the unique continuous function
with $C_{\mu}(0)=0$ such that $\wh \mu(z) = \exp \left ( C_{\mu}(z)\right ), z\in \rd$.

We use the L\'evy-Khintchine triplet $(A, \nu, \gm)$ of $\mu\in I(\rd)$ in the sense that
\begin{equation}
\begin{split}
C_{\mu}(z) = -{2^{-1}} &\langle z,Az \rangle  + {\rm i} 
\langle \gm,z \rangle\\
&+ \int_{\rd}\left(e^{{\rm i} \langle z,x \rangle }-1-
{{\rm i} \langle z,x \rangle}(1+|x|^2)^{-1}\right)\nu(dx),\,z\in\rd,
\end{split}
\end{equation}
where $A$ is a symmetric nonnegative-definite $d \times d$ matrix, 
$\gamma\in\rd$ and $\nu$ is a measure  (called the L\'evy measure) on 
$\rd$ satisfying
\begin{equation*}
\nu(\{0\}) = 0 \,\,\text{and}\,\,\int_{\rd} (|x|^{2} \wedge 1) \nu(dx) < \infty.
\end{equation*}
The following is a basic result on the L\'evy measure of $\mu\in I(\rd)$.

\begin{prop}\label{p2.1}
{\rm (Polar decomposition of L\'evy measures.)
([9], [3])}
Let $\nu$ be the L\'evy measure of the characteristic function of some 
$\mu\in I(\R^d)$ with $0<\nu(\R^d)\les\infty$. Then
there exist a measure $\ld$ on $S=\{\xi\in\rd : |\xi |=1\}$
with $0<\ld(S)\les\infty$ and
a family $\{\nu_{\xi}\colon \xi\in S\}$ of measures on $(0,\infty)$ such that
$\nu_{\xi}(B)$ is measurable in $\xi$ for each $B\in\mathcal B((0,\infty))$,
$0<\nu_{\xi}((0,\infty))\les\infty$ for each $\xi\in S$, and
$$
\nu(B)=\int_S \ld(d\xi)\int_0^{\infty} 1_B(r\xi)\nu_{\xi}(dr),\quad
 B\in \mathcal B (\mathbb R^d \setminus \{ 0\}).
$$
Here $\ld$ and $\{\nu_{\xi}\}$ are uniquely determined by $\nu$
up to multiplication of measurable functions $c(\xi)$ and ${c(\xi)^{-1}}$,
respectively, 
with $0<c(\xi)<\infty$, and  $\nu_{\xi}$ is called the radial component of $\nu$.
We call $(\ld (d\xi), \nu_{\xi}(dr))$ a polar decomposition of the L\'evy measure 
$\nu \ne 0$.
\end{prop}

Five classes in $I(\rd)$ we are going to discuss in this paper are the following.
As mentioned before, they are defined in terms of the radial component $\nu_{\xi}$ of Levy measures.

\n
(1) Class $U(\rd)$ (the Jurek class) :
\begin{equation*}
\nu_{\xi} (dr) = \ell _{\xi}(r)dr,
\end{equation*}
where $\ell_{\xi}(r)$ is measurable in $\xi\in S$ and
decreasing in $r\in (0,\infty)$.
(Here and in what follows, we use the word \lq\lq decreasing" in the non-strict sense.)

\vskip 3mm
\n
(2) Class $B(\rd) $ (the Goldie--Steutel--Bondesson class) :
\begin{equation*}
\nu_{\xi}(dr)=\ell_{\xi}(r)dr,
\end{equation*}
where  $\ell_{\xi}(r)$ is measurable in $\xi\in S$ and completely monotone on $(0,\infty)$.

\vskip 3mm
\n
(3) Class $L(\rd)$ (the class of selfdecomposable distributions) :
\begin{equation*}
\nu_{\xi}(dr)={k_{\xi}(r)}{r^{-1}}dr,
\end{equation*}
where $k_{\xi}(r)$ is measurable in $\xi\in S$ and decreasing on $(0,\infty)$.

\vskip 3mm
\n
(4) Class $T(\rd)$ (the Thorin class) :
\begin{equation*}
\nu_{\xi}(dr)={k_{\xi}(r)}{r^{-1}}dr,
\end{equation*}
where $k_{\xi}(r)$ is measurable in $\xi\in S$ and completely monotone on $(0,\infty)$.

\vskip 3mm
\n
(5) Class $G(\rd) $ (the class of generalized type $G$ distributions) :
\begin{equation*}
\nu_{\xi}(dr)=g_{\xi}(r^2)dr,
\end{equation*}
where
$g_{\xi}(x)$ is measurable in $\xi\in S$ and completely monotone on $(0,\infty)$.
(If $\mu\in G(\rd)$ is symmetric, it is of type $G$ distribution.)

From the definitions, it is trivial that
$$B(\rd)\cup L(\rd) \cup G(\rd)\subset U(\rd)\quad\text{and}\quad T(\rd)\subset L(\rd).$$
Also the same argument as in [2] shows that
$$T(\rd)\subset B(\rd)\subset G(\rd).$$
For that, first note that\\
(1) the product of two completely monotone functions on $(0,\infty)$ is also completely monotone on $(0,\infty)$, 
\\
(2) $f(x) = x^{-\al}, \al >0,$ is  completely monotone  on $(0,\infty)$,\\
and\\
(3) if $\phi$ is a  completely monotone function on $(0,\infty)$ and $\psi$ is a nonnegative differentiable 
function on 
$(0,\infty)$ whose derivative is completely monotone, then the composition $\phi(\psi)$ is completely monotone 
on $(0,\infty)$, (see Corollary 2 in p.\,441 of [4]).\\
It follows from (1) and (2) that $T(\rd)\subset B(\rd)$.
If we put $g_{\xi}(x) = l_{\xi}(x^{1/2}),$ then by (3), we have that $B(\rd)\subset G(\rd).$
These inclusions are all strict, as shown below.
For example, if we take a decreasing but not completely monotone function $l_{\xi}$, then
we can see that $B(\rd)\subsetneqq U(\rd)$.
Also, if we take a completely monotone function $g_{\xi}$ which cannot be expressed as 
$g_{\xi}(x) = l_{\xi}(x^{1/2})$ with some completely monotone function  $l_{\xi}$, then we see
that $B(\rd)\subsetneqq G(\rd)$.
Similar arguments work for all other cases.

Note that there are no relationships of inclusion between the classes $B(\rd)$ and $L(\rd)$,
and the classes $G(\rd)$ and $L(\rd)$.  Actually it is easy to see that
$B(\rd) \setminus L(\rd) \ne \emptyset, L(\rd) \setminus B(\rd) \ne \emptyset, G(\rd) \setminus L(\rd) \ne \emptyset,$
and $L(\rd) \setminus G(\rd) \ne \emptyset$.

These five classes are also characterized by mappings from infinitely divisible 
distributions to infinitely divisible 
distributions defined by the distributions of stochastic integrals with respect to
L\'evy processes.
In what follows, $\{X_s^{(\mu)}\}$ stands  for a L\'evy process on $\rd$ with 
$\law (X_1^{(\mu)}) =\mu$.

\begin{defn}\label{d2.2}

\n
(1) ($\mathcal U$-mapping) {\rm ([5])}
For $\mu\in I(\rd)$,
$$
\mathcal U (\mu) = \law\left (\int_0^1sd\xsm \right ).
$$
(2)
($\Up$-mapping) 
{\rm ([3])}
For $\mu\in I(\rd)$,
$$
\Up (\mu) = \law\left (\int_0^1\log(s^{-1})d\xsm \right ).
$$
(3)
($\Phi$-mapping) 
For $\mu\in I_{\log}(\rd)$,
$$
\Phi (\mu) = \law\left (\int_0^{\infty}e^{-s}d\xsm \right ).
$$
(4)
($\Psi$-mapping) ([3])
Let $e(t) = \int_t^{\infty}{e^{-u}}{u^{-1}}du$, $t>0$,  and denote its inverse function
by $e^*(s)$.
For $\mu\in I_{\log}(\rd)$,
$$
\Psi (\mu) =\law\left ( \int _0^{\infty} e^*(s) d\xsm\right ).
$$
(5) ($\mathcal G$-mapping)
Let 
$h(t) = \int_t^{\infty} e^{-u^2} du$,  $t >0$, and 
denote its  inverse function by $h^*(s)$.
For $\mu \in I(\rd)$,
\begin{equation*}
\mathcal G (\mu) = \law \left (\int_0^{\sqrt{\pi}/2}h^*(s) dX_s ^{(\mu)} \right ).
\end{equation*}
\end{defn}

\begin{rem}
Letting $\mfr D$ denote the domain, we have $\mfr D(\mcal U)=\mfr D(\Up)
=\mfr D(\mcal G)=I(\rd)$ and $\mfr D(\Ph)=\mfr D(\Ps)=I_{\log}(\rd)$.
Recall that $\mfr D(\Ph)$ and $\mfr D(\Ps)$ are defined as the class of 
$\mu\in I(\rd)$ such that $\int_0^t e^{-s}d\xsm$ and $\int_0^t e^*(s)d\xsm$,
respectively, are convergent in probability as $t\to\infty$. 
For two mappings $\Ph_1$ and $\Ph_2$, the composition $\Ph_2 \Ph_1$ are 
defined as $\mfr D(\Ph_2 \Ph_1)=\{\mu\colon \mu\in \mfr D(\Ph_1)\text{ and }
\Ph_1(\mu)\in\mfr D(\Ph_2)\}$ and $(\Ph_2 \Ph_1)(\mu)=\Ph_2(\Ph_1(\mu))$ for
$\mu\in \mfr D(\Ph_2 \Ph_1)$. 
It is known that $\Psi=\Up\Phi=\Phi\Up$ ([3]),
where the equality of the domains is also implied.
\end{rem}

The following are characterizations of the classes in the previous section in terms of
the mappings above, or equivalently, in terms of stochastic integrals with respect to L\'evy processes.

\begin{thm}
\par
\n
$(1)$ $U(\rd) = \mathcal U (I(\rd))$. {\rm ([5])}\\
$(2)$ $B(\rd) = \Up (I(\rd))$. {\rm ([3])}\\
$(3)$ $L(\rd) = \Phi(I_{\log}(\rd))$. {\rm ([18] and others.) }\\
$(4)$ $T(\rd)= \Psi(I_{\log}(\rd))$. {\rm ([3])}\\
$(5)$ $G(\rd) = \mathcal G (I(\rd))$. 
\end{thm}

\begin{rem}
In [1], the equality (5) is proved within $I_{\rm sym}(\rd)$.
However, the same proof works for proving (5).
\end{rem}

We now define the nested subclasses of the 
five classes above by iterating the respective mappings.

Let $U_0(\rd) = U(\rd)$,  $B_0(\rd) = B(\rd)$, $L_0(\rd)=L(\rd)$, 
$T_0(\rd) =T(\rd)$ and $G_0(\rd) = G(\rd)$.
In the following, the $m$-th power of a mapping denotes $m$ times composition
of the mapping, with the domain being the class of all $\mu$ for which the 
$m$-th power is definable.

\begin{defn}
For $m=0, 1, 2, ...$, let

\n
(1) $U_m(\rd) = {\mathcal U}^{m+1}(I(\rd))$,

\n
(2)
$B_m(\rd) = {\Up}^{m+1}(I(\rd))$,

\n
(3)
$L_m(\rd) = {\Phi}^{m+1}(I_{\log ^{m+1}}(\rd))$,

\n
(4)
$T_m(\rd) = {\Psi}^{m+1}(I_{\log ^{m+1}}(\rd))$,

\n
and

\n
(5)
$G_m(\rd) = {\mathcal G}^{m+1}(I(\rd))$,

\n
and let
$U_{\infty}(\rd)=\bigcap_{m=0}^{\infty}U_m(\rd)$, 
$B_{\infty}(\rd)=\bigcap_{m=0}^{\infty}B_m(\rd)$, 
$L_{\infty}(\rd)=\bigcap_{m=0}^{\infty}L_m(\rd)$, 
$T_{\infty}(\rd)=\bigcap_{m=0}^{\infty}T_m(\rd)$,
and
$G_{\infty}(\rd)=\bigcap_{m=0}^{\infty}G_m(\rd)$.
\end{defn}

\begin{rem}
\par
\n
(1)
The original definition of $L_{m}(\rd)$ in [17] and 
[10] is different from ours, but it is known that
they are the same.  See [15] or Lemma 4.1 of [3].\\
(2)
We have $\mfr D(\Ph^{m+1})=\mfr D(\Ps^{m+1})=I_{\log ^{m+1}}(\rd)$.
This will be shown in the proof of Lemma \ref{l3.8}.\\
(3)
$T_m(\rd)$ here is different from that in [3], where $T_m(\rd)$
denotes $\Up(L_m(\rd))$.\\
(4)
$G_m(\rd)$ here is different from that in [1], where everything is discussed within $I_{\rm sym}(\rd)$.
\end{rem}

As mentioned in the Introduction, the following are known.

\begin{prop}\label{p2.8} {\rm ([6]) }
$U_{\infty}(\rd)= L_{\infty}(\rd)$.
\end{prop}

\begin{rem}
This is Corollary 7 of [6].  
However, as Jurek's proof is not easy for us to follow, 
we will give in the last section of this paper an alternative proof 
for that $U_{\infty}(\rd)\subset L_{\infty}(\rd)$
which directly uses our polar decomposition of L\'evy measures 
and shows that a representation of the L\'evy measure of $\mu$ in $U_{\infty}(\rd)$
is exactly the same as that of $\mu$ in $L_{\infty}(\rd)$
shown in [10].
It is noted that our proof also depends on Jurek's basic idea.
\end{rem}

\begin{prop}\label{p2.10} {\rm ([17] and [10])}
$L_{\infty}(\rd) = \overline{S(\rd)}$,
where the closure is taken under weak convergence and convolution.
\end{prop}

Our main result in this paper is the following.

\begin{thm}\label{t1}
$$
U_{\infty}(\rd) = B_{\infty}(\rd) = L_{\infty}(\rd) = T_{\infty}(\rd) 
= G_{\infty}(\rd) = \overline{S(\rd)}.
$$
\end{thm}

Except $T_{\infty}(\rd)$, we
 will prove this theorem from a more general result (Theorem \ref{t2}), where a
sufficient condition for the limit of the nested subclasses of
a class to be equal to $U_{\infty}(\rd)$ is given.
For $Y_{\infty}(\rd)$, we will show that $T_{\infty}(\rd)= L_{\infty}(\rd)$.
\vskip 10mm

\section{Proof of the main theorem (Theorem \ref{t1})}

In order to prove our main theorem (Theorem \ref{t1}), we need several preparations.

\begin{defn}\label{d3.1}
A class $M$ of distributions on $\rd$ is said to be 
{\it completely closed in the strong sense} (c.c.s.s.),
if $M \subset I(\rd)$ and if the following are satisfied.\\
(1) It is closed under convolution.\\
(2) It is closed under weak convergence.\\
(3) If $X$ is an $\rd$-valued random variable with $\law (X)\in M$, then 
$\law (cX+b)\in M$ for any $c>0$ and $b\in\rd$.\\
(4) $\mu\in M$ implies $\mu ^{s*} \in M$ for any $s>0$, where $\mu ^{s*}$ is 
the distribution
with the characteristic function $(\wh\mu(z))^s$.
\end{defn}

\begin{prop}\label{p3.2}
Fix $0<a<\infty$.
Suppose that $f$ is square integrable on $(0,a)$ and $\int_0^a f(s)ds\neq0$.
Define a mapping $\Phi_f$ by
$$
\Phi_f (\mu) = \law \left (\int_0^a f(s) dX_s^{(\mu)}\right ).
$$
Then the following are true.

\n
$(1)$ $\mfr D(\Ph_f)=I(\rd)$.\\
$(2)$ For all $\mu\in I(\rd)$, $\int_0^a |C_{\mu}(f(s)z)|ds<\infty$
and $C_{\Phi_f  (\mu)}(z) = \int_0^a C_{\mu}(f(s)z)ds$.\\
$(3)$ If $M$ is c.c.s.s., then 
$\Phi_f (M) \subset M$.\\
$(4)$ If $M$ is c.c.s.s., then $\Phi_f  (M)$ is also c.c.s.s.
\end{prop}

\n
{\it Proof.} Define $\wt f(s)$ as $\wt f(s)=f(s)$ for $s\in(0,a)$ and 
$\wt f(s)=0$ for $s\in [0,\infty)\setminus(0,a)$.
Since  $\wt f(s)$ is locally square integrable on $[0,\infty)$,
 $\int_B \wt f(s) dX_s^{(\mu)}$ is definable for all 
$\mu\in I(\rd)$ and all bounded Borel sets $B$ in $[0,\infty)$ 
by Proposition 3.4 of [13]. Then $\Ph_f(\mu)$ is the law of 
$\int_{(0,a)}\wt f(s)dX_s^{(\mu)}$. 
Hence (1) is true. (2) is a consequence of Proposition 2.17 of [13].

Proof of (3). Suppose that $M$ is c.c.s.s.\ and $\mu\in M$. We recall the definition 
of $\int_B \wt f(s) dX_s^{(\mu)}$ 
 in Sato [12] or [13]. 
A function $g(s)$ is
called a simple function if $g(s)=\sum_{j=1}^n b_j 1_{B_j}(s)$ for some $n$, 
where $B_1,\ldots,
B_n$ are disjoint Borel sets in $[0,\infty)$ and $b_1,\ldots,b_n\in \R$. For
such a simple function we define $\int_B g(s) dX_s^{(\mu)}=\sum_{j=1}^n 
b_j X^{(\mu)}(B\cap B_j)$ for any bounded Borel set $B$ in $[0,\infty)$, 
using the $\rd$-valued
independently scattered random measure $X^{(\mu)}$ induced by the process
$X_s^{(\mu)}$.
In our case the law of $Y=\int_B g(s) dX_s^{(\mu)}$ belongs to $M$, since
\[
C_{\law (Y)}(z)=\sum_{j=1}^n \int_{B\cap B_j} C_{\mu}(b_j z)ds
=\sum_{j=1}^n C_{\mu}(b_j z)\,\mrm{Leb} (B\cap B_j),
\]
where Leb denotes Lebesgue measure.
Definability of $\int_B \wt f(s) dX_s^{(\mu)}$ mentioned in the proof of (1) 
means that there are
simple functions $g_k(s)$, $k=1,2,\ldots$, such that $g_k(s)\to \wt f(s)$ a.e.\ 
as $k\to\infty$ and that,
for all bounded Borel sets $B$, $\int_B g_k(s)dX_s^{(\mu)}$ converges in 
probability to $\int_B \wt f(s) dX_s^{(\mu)}$ as $k\to\infty$.
Since $M$ is closed under weak convergence, it follows that $\Ph_f(\mu)\in M$.

Proof of (4). Suppose that $M$ is c.c.s.s.  
If $\mu_1$ and $\mu_2$ are in $M$, then $\Ph_f(\mu_1)*\Ph_f(\mu_2)
=\Ph_f(\mu_1*\mu_2)\in \Ph_f(M)$. Hence $\Ph_f(M)$ is closed under convolution.
For $c>0$ and $b\in\rd$, we have
\[
\int_0^a f(s)d(cX_s^{(\mu)}+bs)=c\int_0^a f(s)dX_s^{(\mu)}+b\int_0^a f(s)ds.
\]
Since $\int_0^a f(s)ds\neq0$, it follows from this that $\Ph_f(M)$ has 
property (3) of Definition \ref{d3.1}. We have, for $t>0$,
\[
tC_{\Ph_f(\mu)}(z)=t\int_0^a C_{\mu}(f(s)z)ds=\int_0^a C_{\mu^{t*}}(f(s)z)ds.
\]
Hence $\Ph_f(M)$ has 
property (4) of Definition \ref{d3.1}. 
 It remains to prove that $\Ph_f(M)$ is closed under weak 
convergence. We make use of the following fact for $\mu_n$, $n=1,2\ldots$, and
$\mu$ in $I(\rd)$:
\begin{equation}\label{cont}
\text{if $\mu_n\to\mu$, then $\Ph_f(\mu_n)\to\Ph_f(\mu)$.}
\end{equation}
To show this, let $\mu_n\to\mu$ and recall that
\begin{align*}
C_{\Ph_f(\mu_n)}(z)&=\int_0^a C_{\mu_n}(f(s)z)ds,\\
C_{\Ph_f(\mu)}(z)&=\int_0^a C_{\mu}(f(s)z)ds.
\end{align*}
and that
\[
C_{\mu_n}(f(s)z)\to C_{\mu}(f(s)z).
\]
Hence it is enough to show the existence of an integrable function $h(s)$
on $(0,a)$ such that $\sup_n |C_{\mu_n}(f(s)z)|\les c_z h(s)$
with constant $c_z$ depending only on $z$ and to use the
dominated convergence theorem.
Let $(A_n,\nu_n,\gm_n)$ be the triplet of $\mu_n$. Since $\mu_n$ is convergent,
we have
{\allowdisplaybreaks
\begin{gather}
\sup_n \tr A_n <\infty,\label{pc1}\\
\sup_n \int_{\rd}(|x|^2\land1)\nu_n(dx)<\infty,\label{pc2}\\
\sup_n |\gm_n|<\infty.\label{pc3}
\end{gather}
We have
\[
|C_\mu (z)|\les \frac12 (\tr\,A)|z|^2 +|\gm| |z|+
\int |g(z,x)|\nu(dx)
\]
with $g(z,x)=e^{i\langle z,x\rangle} -1-i\langle z,x\rangle/(1+|x|^2)$. Hence
\begin{align*}
|C_{\mu_n} (f(s)z)|&\les \frac12 (\tr\,A_n)|f(s)z|^2
+|\gm_n| |f(s)z|\\
&\quad+\int_{\R^d} |g(z,f(s)x)| \nu_n(dx)+\int_{\R^d} |g(f(s)z,x)-
g(z,f(s)x)| \nu_n(dx)\\
&=I_1+I_2+I_3+I_4\qquad\text{(say)}.
\end{align*}
Let $c'_z, c''_z,\ldots$ denote constants depending on $z$.
It follows from \eqref{pc1} and \eqref{pc3} that $I_1+I_2\les
c'_z(f(s)^2+|f(s)|)$. Since
$|g(z,x)|\les c''_z |x|^2 / (1+|x|^2)$, it follows from \eqref{pc2} that
\begin{align*}
I_3&\les c''_z\int_{\R^d}\frac{|f(s)x|^2}{1+|f(s)x|^2} \nu_n(dx)\\
&\les c''_z\left( f(s)^2\int_{|x|\les1}|x|^2\nu_n(dx)+
\int_{|x|>1} \nu_n(dx)\right)\\
&\les c'''_z (f(s)^2+1).
\end{align*}
Further, using 
\[
|g(uz,x)-g(z,ux)| \les |z|\frac{|x|^3(|u|+|u|^3)}{(1+|x|^2)(1+|ux|^2)}
\qquad\text{for $u\in\R$},
\]
we obtain, with $u=f(s)$,
\begin{align*}
I_4 &\les |z|\int_{\R^d}
\frac{|x|^3(|u|+|u|^3)}{(1+|x|^2)(1+|ux|^2)}\nu_n(dx)\\
&\les|z|\int_{|x|\les1}\left(|x|^3 |u|+\frac{|x|^2}{2}u^2\right)
\nu_n(dx)
+|z|\int_{|x|>1}\left(\frac{|ux|}{1+|ux|^2}
+\frac{|ux|}{1+|x|^2}\right)\nu_n(dx)\\
&\les|z|\int_{|x|\les1}\left(|x|^3 |u|+\frac{|x|^2}{2}u^2\right)
\nu_n(dx)+|z|\int_{|x|>1}\left(\frac{1}{2}+\frac{|u|}{2}\right)
\nu_n(dx)\\
&\les c''''_z (f(s)^2+|f(s)|+1)
\end{align*}
from \eqref{pc2}. Thus we get $h(s)$ as asserted. This proves \eqref{cont}.}
Now, let $\wt\mu_1, \wt\mu_2, \ldots$ be in $\Ph_f(M)$ and tend to $\wt\mu$.
For each $\wt\mu_n$ we can find $\mu_n$ such that $\wt\mu_n=\Ph_f(\mu_n)$.
Let $(\wt A_n,\wt\nu_n,\wt\gm_n)$ and $(A_n,\nu_n,\gm_n)$ be the triplets of 
$\wt\mu_n$ and $\mu_n$, respectively.
We claim that $\{\mu_n\colon n=1,2,\ldots\}$ is precompact, which is 
equivalent to \eqref{pc1}, \eqref{pc2}, \eqref{pc3}, plus
\begin{equation}\label{pc4}
\lim_{l\to\infty}\sup_n \int_{|x|>l} \nu_n(dx)=0
\end{equation}
(see p.\,13 of [3]). 
Since $\{\wt\mu_n\}$ is precompact, \eqref{pc1}--\eqref{pc4} hold for
$(\wt A_n,\wt\nu_n,\wt\gm_n)$ in place of $(A_n,\nu_n,\gm_n)$.
Recall that
{\allowdisplaybreaks
\begin{gather}
\wt A_n=\left(\int_0^a f(s)^2 ds\right) A_n,\label{rep1}\\
\wt\nu_n(B)=\int_0^a ds\int_{\rd} 1_B(f(s)x)\nu_n(dx),\qquad B\in
\mcal B(\rd\setminus\{0\}),\label{rep2}\\
\wt\gm_n=\int_0^a f(s)ds\left(\gm_n+\int_{\rd}x\left(\frac{1}{1+|f(s)x|^2}
-\frac{1}{1+|x|^2}\right)\nu_n(dx)\right),\label{rep3}
\end{gather}
(see Proposition 2.6 of [14]). 
Hence we obtain \eqref{pc1}. To see \eqref{pc2},
note that
\[
\int_{\rd}(|x|^2\land1)\wt\nu_n(dx)=\int_0^a ds\int_{\rd}(|f(s)x|^2\land1)
\nu_n(dx)
\]
and consider two cases 
separately: (1) there is $c>0$ such that $|f(s)|\in\{0,c\}$ for a.e.\ 
$s\in(0,a)$; (2) otherwise.  In case (1) we have
\begin{align*}
&\int_{\rd}(|x|^2\land1)\wt\nu_n(dx)=\int_{|f(s)|=c}ds\int_{\rd}
(|cx|^2\land1)\nu_n(dx)\\
&\qquad\ges(c^2\land1)\int_{|f(s)|=c}ds\int_{\rd}
(|x|^2\land1)\nu_n(dx).
\end{align*}
In case (2), choosing $c>0$ such that $\int_{|f(s)|\les c}ds>0$
and $\int_{|f(s)|>c}ds>0$, we have
\begin{align*}
&\int_{\rd}(|x|^2\land1)\wt\nu_n(dx)=\int_0^a ds\int_{|f(s)x|\les1}
|f(s)x|^2 \nu_n(dx)+\int_0^a ds\int_{|f(s)x|>1} \nu_n(dx)\\
&\qquad\ges\int_{|f(s)|\les c}f(s)^2ds\int_{|x|\les 1/c}|x|^2
\nu_n(dx)+\int_{|f(s)|> c}ds\int_{|x|> 1/c}\nu_n(dx).
\end{align*}
Hence we obtain \eqref{pc2} in any case. To prove \eqref{pc4}, choose $c>0$
with $\int_{|f(s)|>c}ds>0$ and note that
\[
\int_{|x|>l}\wt\nu_n(dx)=\int_0^a ds\int_{|f(s)x|>l}\nu_n(dx)
\ges \int_{|f(s)|>c}ds \int_{|x|>l/c} \nu_n(dx).
\]
In order to obtain \eqref{pc3}, it suffices to show the boundedness
of
\[
\int_0^a f(s)ds\int_{\rd}x\left(\frac{1}{1+|f(s)x|^2}
-\frac{1}{1+|x|^2}\right)\nu_n(dx)
\]
since we have \eqref{rep3} and $\int_0^a f(s)ds\neq0$. This
boundedness is true because
\begin{align*}
&\int_0^a |f(s)|ds\int_{\rd}|x|\left|\frac{1}{1+|f(s)x|^2}
-\frac{1}{1+|x|^2}\right|\nu_n(dx)\\
&\qquad\les\int_0^a |f(s)|ds\int_{\rd}\frac{|x|(|f(s)x|^2+|x|^2)}
{(1+|f(s)x|^2)(1+|x|^2)}\nu_n(dx)\\
&\qquad\les\int_0^a ds\left(\frac{f(s)^2}{2}+|f(s)|\right)
\int_{|x|\les1}|x|^2\nu_n(dx)+\int_0^a ds\left(\frac{|f(s)|}{2}+\frac{1}{2}
\right)\int_{|x|>1}\nu_n(dx)\\
&\qquad\les\mrm{const}\int_{\rd}(|x|^2\land1)\nu_n(dx).
\end{align*}
Thus we have proved that} $\{\mu_n\}$ is precompact.
Therefore there exists a subsequence $\{\mu_{n_k}\}$ convergent
to some $\mu\in M$. It follows from \eqref{cont} that
$\Ph_f(\mu_{n_k})\to\Ph_f(\mu)$. Hence $\wt\mu=\Ph_f(\mu)$, concluding
$\wt\mu\in \Ph_f(M)$. Therefore $\Ph_f(M)$ is closed under weak 
convergence, which completes the proof of Proposition \ref{p3.2}.
\qed

\vskip 3mm

\begin{rem}\label{r3.3}
(1) Note that Proposition \ref{p3.2} can be applied to $\Up$- and 
$\mathcal G$-mappings, because in those mappings
the upper limit of the stochastic integral is finite, $f$ is square integrable,
and $\int_0^a f(s)ds\neq0$.\\
(2) Proposition \ref{p3.2} (4) is not necessarily true when $a=\infty$.
Namely, there is a mapping $\Phi_f $  defined by 
$\Phi_f  (\mu) = \law\left (\sek f(s)dX_s^{(\mu)}\right )$ such that 
$\Phi_f  (M\cap\mfr D(\Phi_f))$ is not closed under weak convergence
for some $M$ which is c.c.s.s. 
Indeed, let $\Ph_f=\Ps_{\al}$ with $0<\al<1$, which is defined similarly
to Example \ref{e1} (3).  Looking at Theorem 4.2 of [14], let
$\mu_n$ be such that 
\[
C_{\mu_n}(z)=\int_S \ld(d\xi)\int_0^{\infty}(e^{i\langle z,r\xi\rangle} -1)
r^{-\al-1} e^{-r/n} dr,
\]
where $\ld$ is a finite nonzero measure on $S$. Then $\mu_n\in
\Phi_f  (\mfr D(\Phi_f))$ and $\mu_n$ tends to an $\al$-stable 
distribution $\mu$ as $n\to\infty$, 
but $\mu\not\in\Phi_f  (\mfr D(\Phi_f))$ again by 
Theorem 4.2 of [14].
Thus $\Ph_f(\mfr D(\Ph_f))$ 
is not closed under weak convergence. \\
(3) However, it is known that when $\Phi_f  = \Phi$, Proposition \ref{p3.2} 
(3) and (4) are true
with $\Ph_f(M)$ replaced by $\Ph(M\cap\mfr D(\Ph))$, even if $a=\infty$.
See Lemma 4.1 of [3].
In particular, $L_m(\rd)$ is c.c.s.s.\ for $m=0,1,\ldots$.
\end{rem}

We are now going to prove the following.

\begin{thm}\label{t2}
Let $0<t_0\les\infty$.  Let $p(u)$ be a positive decreasing function on 
$(0,t_0)$ such that $\int_0^{t_0}(1+u^2)p(u)du<\infty$.
Define $g(t)=\int_t^{t_0} p(u)du$ for $0<t<t_0$ and $s_0=g(0+)<\infty$.
Let $t=f(s)$, $0<s<s_0$, be the inverse function of $s=g(t)$, $0<t<t_0$.
Define 
\[
\Ph_f(\mu)=\law\left( \int_0^{s_0} f(s)dX_s^{(\mu)}\right)\qquad\text{for 
$\mu\in\mfr{D}(\Ph_f)$}.
\]
Then,\\
{\rm(1)}\qquad $\mfr{D}(\Ph_f)=I(\R^d)$.\\
{\rm(2)}\qquad $I(\R^d)\supset \Ph_f(I(\R^d))\supset \Ph_f^2(I(\R^d))\supset
\cdots$.\\
{\rm(3)}\qquad $\Ph_f^{m}(I(\R^d))\subset U_{m-1}(\R^d)$ for $m=1,2,\ldots$.\\
{\rm(4)}\qquad $\Ph_f(S_{\al}(\R^d))=S_{\al}(\R^d)$ for\/ $0<\al\les2$.\\
{\rm(5)}\qquad $\bigcap_{m=1}^{\infty}\Ph_f^m (I(\R^d))=U_{\infty}(\R^d)$.
\end{thm}

\begin{ex}\label{e1}  The following are examples of $\Ph_f$ in 
Theorem \ref{t2}.\\
(1) $\Ph_f=\Up$ if $p(u)=e^{-u}$, $g(t)=e^{-t}$ with
$t_0=\infty$, and $f(s)=\log(s^{-1})$ with $s_0=1$.\\
(2) $\Ph_f=\mcal G$ if $p(u)=e^{-u^2}$ with $t_0=\infty$ and $s_0=\sqrt{\pi}/2$.\\
(3) $\Ph_f=\Ps_{\al}$ with $-1\les\al<0$ in [14]
if $p(u)=u^{-\al-1} e^{-u}$ with $t_0=\infty$ and $s_0=\Gm(-\al)$. Note that
$\Ps_{-1}=\Up$.\\
(4) $\Ph_f=\Ph_{\bt,\al}$ with $-1\les\al<0$ and $\bt\les\al-1$ in [14]
if 
$$p(u)=(\Gm(\al-\bt))^{-1}(1-u)^{\al-\bt-1}u^{-\al-1}$$
 with $t_0=1$ and $s_0=
\Gm(-\al)/\Gm(-\bt)$.\\
(5) $\Ph_f$ with $f(s)=(1+\al s)^{1/(-\al)}$, $-1\les\al<0$, and $s_0=1/(-\al)$.
This is a special case of (4) with $\bt=\al-1$ and $g(t)=(1-t^{-\al})/(-\al)$.\\
(6) $\Ph_f$ with $f(s)=1-(\Gm(-\bt)s)^{1/(-\bt-1)}$, $\bt\les -2$, and
$s_0=1/\Gm(-\bt)$.  This is another special case of (4) with $\al=-1$ and
$g(t)=(1-t)^{-\bt-1}/\Gm(-\bt)$.\\
See p.\,49 of [14] for (5) and (6). 
In particular, $\Ph_{-2,-1}=\mcal U$,
because in this case 
$p(u)=1$, $g(t)=1-t$, $f(s)=1-s$, and $\int_0^1C_{\mu}((1-s)z)ds= 
\int_0^1C_{\mu}(sz)ds$.
\end{ex}

\vskip 3mm
In order to prove Theorem \ref{t2}, we need two lemmas.

\begin{lem}\label{l2}
For $j=0,1$ let $0<s_j<\infty$ and $f_j(s)$ be a square integrable function on
$(0,s_j)$. Let
\[
\Ph_{f_j}(\mu)=\law\left( \int_0^{s_j} f_j(s)dX_s^{(\mu)}\right)\qquad\text{for 
$\mu\in\mfr{D}(\Ph_{f_j})=I(\R^d)$}.
\]
Then
\begin{equation}\label{1}
\Ph_{f_1}(\Ph_{f_0}(\mu))=\Ph_{f_0}(\Ph_{f_1}(\mu))\qquad\text{for 
$\mu\in I(\R^d)$}.
\end{equation}
\end{lem}

\begin{proof}
We can check that
\begin{equation}\label{2}
\int_0^{s_1} du \int_0^{s_0} |C_{\mu}(f_0(s)f_1(u)z)|ds<\infty\qquad\text{for 
$z\in\R^d$},
\end{equation}
because, as in the proof of Proposition \ref{p3.2} (4),
\[
|C_{\mu}(f_0(s)f_1(u)z)|\les c_z((f_0(s)f_1(u))^2+|f_0(s)f_1(u)|+1),
\]
where $c_z$ is a constant depending only on $z$.
By virtue of \eqref{2}, we can apply Fubini's theorem and
{\allowdisplaybreaks
\begin{align*}
&C_{\Ph_{f_1}(\Ph_{f_0}(\mu))}(z)=\int_0^{s_1}C_{\Ph_{f_0}(\mu)}(f_1(u)z)du
=\int_0^{s_1}du\int_0^{s_0}C_\mu (f_0(s)f_1(u)z)ds\\
&\quad=\int_0^{s_0}ds\int_0^{s_1}C_{\mu}(f_1(u)f_0(s)z)du
=\int_0^{s_0}C_{\Ph_{f_1}(\mu)}(f_0(s)z)ds=C_{\Ph_{f_0}(\Ph_{f_1}(\mu))}(z),
\end{align*}
that is}, \eqref{1} holds.
\end{proof}

\begin{lem}\label{l3}
Let $0<s_0<\infty$. Let $f(s)$ be a nonnegative, square integrable function on 
$(0,s_0)$ such that $\int_0^{s_0} f(s) ds>0$. Then
\begin{equation*}\label{3}
\Ph_f(S_{\al}(\R^d))=S_{\al}(\R^d)\qquad\text{for $0<\al\les2$}.
\end{equation*}
\end{lem}

\begin{proof}
A distribution $\mu$ is in $S_{\al}(\R^d)$ if and only if $\mu\in I(\rd)$
and for any $c>0$ there is $\gm_c\in\R^d$ such that
\[
\wh\mu(cz)=\wh\mu(z)^{c^{\al}} e^{i\langle\gm_c,z\rangle},\qquad z\in\R^d,
\]
that is,
\[
C_{\mu}(cz)=c^{\al}C_{\mu}(z)+i\langle\gm_c,z\rangle.
\]
For $c=0$ this is trivially true with $\gm_0=0$. 
If $\mu\in S_{\al}(\rd)$, then
{\allowdisplaybreaks
\begin{align*}
&C_{\Ph_f(\mu)}(cz)=\int_0^{s_0}C_{\mu}(f(s)cz)ds=
\int_0^{s_0}c^{\al}C_{\mu}(f(s)z)ds+\int_0^{s_0}i\langle\gm_c,f(s)z\rangle ds\\
&\qquad=c^{\al}C_{\Ph_f(\mu)}(z)+i\int_0^{s_0}f(s)ds\langle\gm_c,z\rangle,
\end{align*}
which shows that} $\Ph_f(\mu)\in S_{\al}(\rd)$. Further, if $\mu\in S_{\al}(\rd)$, then
\begin{equation*}
C_{\Ph_f(\mu)}(z)=\int_0^{s_0}C_{\mu}(f(s)z)ds=\left(\int_0^{s_0}f(s)^{\al}ds
\right)C_{\mu}(z)+i\left\langle\int_0^{s_0} \gm_{f(s)}ds,z\right\rangle.
\end{equation*}
Recall (E\,18.6 of [11]) that, if $\mu\in S_{\al}(\rd)$, then there is
$\ta\in\R^d$ such that 
\begin{equation*}\label{5}
\gm_c=\begin{cases}(c-c^{\al})\ta\qquad& \text{for $\al\neq1$,}\\
-c(\log c )\ta\qquad& \text{for $\al=1$,} \end{cases}
\end{equation*}
which shows that $\int_0^{s_0}|\gm_{f(s)}|ds<\infty$.

Conversely, suppose that $\wt\mu\in S_{\al}(\rd)$ with $\wt\gm_c$ in place of $\gm_c$.
Choose
\[
\gm=-\left(\int_0^{s_0} f(s)ds\right)^{-1} \left(\int_0^{s_0}f(s)^{\al}ds\right)^{-1}
\int_0^{s_0} \wt\gm_{f(s)} ds.
\]
Let $\mu\in I(\R^d)$ be such that
\[
C_{\mu}(z)=\left(\int_0^{s_0}f(s)^{\al}ds\right)^{-1}C_{\wt\mu}(z)+i\langle
\gm,z\rangle.
\]
Then $\mu\in S_{\al}(\rd)$ and 
{\allowdisplaybreaks
\begin{align*}
C_{\Ph_f (\mu)}(z)&=\int_0^{s_0} C_{\mu}(f(s)z)ds\\
&=\left(\int_0^{s_0}f(s)^{\al}ds\right)^{-1}\int_0^{s_0} C_{\wt\mu}(f(s)z)ds
+i\int_0^{s_0}\langle\gm, f(s)z\rangle ds\\
&=C_{\wt\mu}(z)+i\left(\int_0^{s_0}f(s)^{\al}ds\right)^{-1}
\int_0^{s_0}\langle\wt\gm_{f(s)},z\rangle ds+i\int_0^{s_0}f(s)ds \langle
 \gm,z\rangle\\
&=C_{\wt\mu}(z),
\end{align*}
and hence} $\Ph_f (\mu)=\wt\mu$.
This completes the proof.
\end{proof}

\begin{proof}[Proof of Theorem \ref{t2}] 
In the following, we write $I$ for $I(\rd)$ for simplicity.

(1) Since $s_0<\infty$ and since
\[
\int_0^{s_0} f(s)^2ds=\int_0^{t_0} t^2 p(t)dt<\infty,
\]
we have $\mfr{D}(\Ph_f)=I$. See Proposition \ref{p3.2} (1).

(2) It follows from $\Ph_f(I)\subset I$ that
$\Ph_f^2(I)\subset \Ph_f(I)$. Then, $\Ph_f^3(I)\subset \Ph_f^2(I)$,
and so on.

(3) Let $\wt\mu=\Ph_f(\mu)$. Let $\wt\nu$ and $\nu$ be the L\'evy 
measures of $\wt\mu$ and $\mu$, respectively. Let $(\ld(d\xi),\nu_{\xi}(dr))$
be a polar decomposition of $\nu$. We know
\[
\wt\nu(B)=\int_0^{s_0} ds\int_{\R^d}1_B(f(s)x)\nu(dx)
=\int_0^{t_0}p(t)dt\int_{\R^d}1_B(tx)\nu(dx)
\]
for $B\in\mcal B(\R^d)$. If $B=\{r\xi\colon \xi\in D,\;r\in(s,\infty)\}$ 
with $D\in\mcal B(S)$ and $s>0$, then
{\allowdisplaybreaks
\begin{align*}
\wt\nu(B)&=\int_0^{t_0}p(t)dt\int_D\ld(d\xi)\int_{s/t}^{\infty}\nu_{\xi}(dr)
=\int_D\ld(d\xi)\int_{s/t_0}^{\infty}\nu_{\xi}(dr)\int_{s/r}^{t_0} p(t)dt\\
&=\int_D\ld(d\xi)\int_{s/t_0}^{\infty}r^{-1}\nu_{\xi}(dr)\int_s^{rt_0} p(u/r)du\\
&=\int_D\ld(d\xi)\int_s^{\infty}du\int_{u/t_0}^{\infty}p(u/r)r^{-1}\nu_{\xi}(dr).
\end{align*}
Hence, letting $\wt\ld=\ld$ and
\begin{equation*}
\wt l_{\xi}(u)= 
\int_{u/t_0}^{\infty}p(u/r)r^{-1}\nu_{\xi}(dr),
\end{equation*}
we obtain a polar decomposition} $(\wt\ld(d\xi), \wt l_{\xi}(u)du)$ of $\wt\nu$.
Since $p$ is decreasing, $\wt l_{\xi}(u)$ is decreasing in $u$. Therefore
$\wt\mu\in U(\rd)=U_0(\rd)$. Hence $\Ph_f(I)\subset U_0(\rd)=\mcal U(I)$.
 This proves (3) for $m=1$. 

If $\Ph_f^{m}(I)\subset \mcal U^{m}(I)=U_{m-1}(\rd)$, then, using Lemma \ref{l2},
we get
\[
\Ph_f^{m+1}(I)\subset \Ph_f(\mcal U^{m}(I))=\mcal U^{m}(\Ph_f(I))\subset 
\mcal U^{m+1}(I)=U_{m}(\rd).
\]
This completes the induction argument.

(4) Apply Lemma \ref{l3}.

(5) It follows from (3) that
\[
\bigcap_{m=1}^{\infty} \Ph_f^m (I)\subset\bigcap_{m=0}^{\infty} U_m(\rd)
=U_{\infty}(\rd).
\]
On the other hand, it follows from (4) that
\[
S_{\al}(\rd)=\Ph_f^m(S_{\al}(\rd))\subset\Ph_f^m(I).
\]
Hence $S(\rd)=\bigcup_{0<\al\les2}S_{\al}(\rd)\subset\bigcap_{m=1}^{\infty} \Ph_f^m (I)$.
Use Proposition \ref{p3.2} (4) to show that $\Ph_f^m (I)$ is c.c.s.s.
Then we see that $\overline{S(\rd)}\subset\bigcap_{m=1}^{\infty} \Ph_f^m (I)$. Since
$\overline{S(\rd)}=L_{\infty}(\rd)$ (Proposition \ref{p2.10}) and $L_{\infty}(\rd)
=U_{\infty}(\rd)$ 
(Proposition \ref{p2.8}), the proof of (5) is complete.
\end{proof}

We need one more lemma.

\begin{lem}\label{l3.8}
$T_m(\rd)$ is c.c.s.s.\ for $m=0,1,\ldots$\;.
\end{lem}

\begin{proof}
We first show\\
(1) $\mu\in I_{\log ^m}(\rd)$ if and only if $\Up (\mu)\in I_{\log^m}(\rd)$\\
and\\
(2)  $\mu\in I_{\log ^{m+1}}(\rd)$ if and only if $\mu\in I_{\log}(\rd)$ and
$\Phi (\mu)\in I_{\log^{m}}(\rd)$.

Let us prove (1). If $\int_{|y|>1}(\log|y|)^m\nu_{\mu}(dy)<\infty$, then
{\allowdisplaybreaks
\begin{align*}
\int_{|x|>1} & (\log |x|)^m\nu_{\Up(\mu)}(dx)=
\int_{|x|>1} (\log |x|)^m \sek \nu_{\mu}(s^{-1}dx)e^{-s}ds\\
& =\int_{\rd}\nu_{\mu}(dy)\sek (\log |sy|)^m e^{-s} 1_{\{|sy|>1\}}ds\\
& =\int_{|y|>0} \nu_{\mu}(dy)\int_{1/|y|}^{\infty}(\log |sy|)^m e^{-s}ds\\
& =\int_{|y|>0}\nu_{\mu}(dy)\int_{1/|y|}^{\infty}
\sum_{n=0}^m \binom{m}{n}  (\log s)^n(\log|y|)^{m-n} e^{-s}ds\\
& =\sum_{n=0}^m \binom{m}{n} \int_{|y|>1}(\log|y|)^{m-n} \nu_{\mu}(dy)            
\int_{1/|y|}^{\infty}(\log s)^n e^{-s}ds +\text{finite term},
\end{align*}
which implies} $\int_{|x|>1}(\log|x|)^m\nu_{\Up(\mu)}(dx)<\infty$. Here we 
have used that
\[
\int_{1/|y|}^{\infty}(\log s)^n e^{-s}ds \sim (\log(1/|y|))^n e^{-1/|y|},
\qquad |y|\to0,
\]
and that $\int_0^{\infty}(\log s)^n e^{-s}ds$ is finite.
Conversely, if $m\ges1$ and 
$\int_{|x|>1}(\log|x|)^m\nu_{\Up(\mu)}(dx)<\infty$ and if (1) is true for
$m-1$ in place of $m$, then 
$\int_{|y|>1}(\log|y|)^j\nu_{\mu}(dy)<\infty$ for $j=0,\ldots,m-1$, and
the equalities above show that
$\int_{|y|>1}(\log|y|)^m\nu_{\mu}(dy)<\infty$.
As (1) is trivially true for $m=0$, we see that (1) is true for all $m$.

Assertion (2) follows from
{\allowdisplaybreaks
\begin{align*}
\int_{|x|>1} & (\log |x|)^m\nu_{\Phi(\mu)}(dx)
=\int_{|x|>1} (\log |x|)^m \sek \nu_{\mu}(e^sdx)ds\\
& =\int_{\rd}\nu_{\mu}(dy)\sek (\log |e^{-s}y|)^m 1_{\{|e^{-s}y|>1\}}ds\\
& =\int_{|y|>1}\nu_{\mu}(dy)\int_{0}^{\log|y|}(\log |y|- s)^m ds\\
& =(m+1)^{-1}\int_{|y|>1}(\log|y|)^{m+1}\nu_{\mu}(dy).
\end{align*}

It follows from (2) that} $\mfr D(\Ph^m)=I_{\log^m}(\R^d)$.  
Since $\Ps=\Ph\Up=\Up\Ph$, it follows from (1) and (2) that\\
(3) $\mu\in I_{\log ^{m+1}}(\rd)$ if and only if $\mu\in I_{\log}(\rd)$ and
$\Ps (\mu)\in I_{\log^{m}}(\rd)$.\\
Hence $\mfr D(\Ps^m)=I_{\log^m}(\R^d)$. 
Thus, we have
\begin{equation*}
T_m(\rd) = \Psi ^{m+1}(I_{\log^{m+1}}(\rd)) 
= (\Up^{m+1} \Phi^{m+1})(I_{\log^{m+1}}(\rd)),
\end{equation*}
that is,
\begin{equation}\label{3.5}
T_m(\rd) = \Up^{m+1}(L_m(\rd)).
\end{equation}
Since $L_m(\rd)$ is c.c.s.s.\ by Remark 
\ref{r3.3} (3), it follows from Proposition \ref{p3.2} (4) that
$T_m(\rd)$ is c.c.s.s.
\end{proof}

We are now ready to prove Theorem \ref{t1}.

\n
{\it Proof of Theorem \ref{t1}}. 
We have already seen that
$U_{\infty}(\rd) = L_{\infty}(\rd) = \overline{S(\rd)}$.
Since $\Up$- and $\mathcal G$-mappings are examples of $\Phi_f$ in Theorem 
\ref{t2},
it follows from Theorem \ref{t2} (5) that $B_{\infty}(\rd)=
G_{\infty}(\rd)=U_{\infty}(\rd)=\overline{S(\rd)}$.
It remains to show that $T_{\infty}(\rd)=\overline{S(\rd)}$.
It follows from \eqref{3.5}, Proposition \ref{p3.2} (3), (4), and
Remark \ref{r3.3} (3) that $T_m(\rd)=\Up^{m+1}(L_m(\rd))\subset L_m(\rd)$.
Hence
\begin{equation}\label{3.6}
T_{\infty}(\rd) \subset L_{\infty}(\rd).
\end{equation}
It follows from \eqref{3.5} and Lemma \ref{l3} that
$T_m(\rd)=\Up^{m+1}(L_m(\rd))\supset\Up^{m+1}(S(\rd))=S(\rd)$.  Hence we get
$T_m(\rd)\supset \overline{S(\rd)}$ from Lemma \ref{l3.8}. Therefore
\begin{equation}\label{3.7}
T_{\infty}(\rd) \supset \overline{S(\rd)}=L_{\infty}(\rd).
\end{equation}
Thus, \eqref{3.6} and \eqref{3.7} imply that $T_{\infty}(\rd)=L_{\infty}(\rd)$,
which completes the proof of Theorem \ref{t1}.
\qed

\vskip 10mm

\section{Proof of Proposition \ref{p2.8}}

As we announced, we give here our proof of Proposition \ref{p2.8}.
We start with the following fact.

\begin{prop}\label{p7}
Let $\mu\in U_0(\R^d)$ with L\'evy measure $\nu^{\mu}\neq0$. 
Let $\mu=\mcal U(\rh)$ and let $(\ld^{\rh} (d\xi), \nu_{\xi}^{\rh}(dr))$ be a
polar decomposition of the L\'evy measure $\nu^{\rh}$ of $\rh$.
Let $\ld^{\mu}=\ld^{\rh}$ and let
{\allowdisplaybreaks
\begin{gather}
l_{\xi}^{\mu}(r)=\int_r^{\infty} s^{-1}\nu_{\xi}^{\rh}(ds),\qquad r>0,\label{4.1}\\
\nu_{\xi}^{\mu}(dr)=l_{\xi}^{\mu}(r)dr,\qquad r>0,\label{4.2}\\
h_{\xi}^{\rh}(u)=e^{-u}\nu_{\xi}^{\rh}((e^u,\infty)),\qquad u\in\R,\label{4.3}\\
h_{\xi}^{\mu}(u)=e^{-u}\nu_{\xi}^{\mu}((e^u,\infty)),\qquad u\in\R .\label{4.4}
\end{gather}
Then the following are true.\\
$(1)$ $(\ld^{\mu}(d\xi), \nu_{\xi}^{\mu}(dr))$ is a 
polar decomposition of $\nu^{\mu}$.\\
$(2)$  $h_{\xi}^{\mu}(u)$ is absolutely continuous on $\R$ and}
\begin{equation*}
-\frac{d}{du}h_{\xi}^{\mu}(u)=h_{\xi}^{\rh}(u) ,\qquad \text{for}\,\,\mrm{a.e.}\,u\in\R,
\end{equation*}
where $\frac{d}{du}$ denotes Radon--Nikod\'ym derivative.
\end{prop}

\n
{\it Proof}.  It follows from Definition \ref{d2.2} (1) that
\[
\nu^{\mu}(B)=\int_0^1\nu^{\rh}(t^{-1}B)dt,\qquad B\in\mcal B(\R^d\setminus\{0\}).
\]
We have $\nu^{\rh}\neq0$, since $\nu^{\mu}\neq0$.
Let $v>0$ and $D\in\mcal B(S)$. For $B=(v,\infty)D=\{x=r\xi\colon \xi\in D,\,r>v\}$,
notice that
{\allowdisplaybreaks
\begin{align*}
\nu^{\mu}(B)&=\int_0^1dt\int_D\ld^{\rh}(d\xi)\int_{t^{-1}v}^{\infty}\nu_{\xi}^{\rh}(dr)
=\int_D\ld^{\rh}(d\xi)\int_v^{\infty}\nu_{\xi}^{\rh}(dr)\int_{v/r}^1dt\\
&=\int_D\ld^{\rh}(d\xi)\int_v^{\infty}\nu_{\xi}^{\rh}(dr)\int_v^r r^{-1} du
=\int_D\ld^{\rh}(d\xi)\int_v^{\infty}du\int_u^{\infty}r^{-1}\nu_{\xi}^{\rh}(dr)\\
&=\int_S\ld^{\mu}(d\xi)\int_0^{\infty}1_{(v,\infty)D}(r\xi)l_{\xi}^{\mu}(r)dr,
\end{align*}
by \eqref{4.1}.
Thus for a general} $B\in\mcal B(\R^d\setminus\{0\})$, we have
$$
\nu^{\mu}(B)=\int_S\ld_{\xi}^{\mu}(d\xi)\int_0^{\infty}1_{B}(r\xi)
l_{\xi}^{\mu}(r)dr.
$$
Hence (1) is true. Since
\begin{equation*}
h_{\xi}^{\mu}(u)=e^{-u}\int_{e^u}^{\infty}l_{\xi}^{\mu}(s)ds,\qquad u\in\R ,
\end{equation*}
absolute continuity of $h_{\xi}^{\mu}(u)$ is obvious. We have, for a.e. $u\in\R$,
{\allowdisplaybreaks
\begin{align*}
-\frac{d}{du}h_{\xi}^{\mu}(u)
&= e^{-u}\int_{e^u}^{\infty}l_{\xi}^{\mu}(s)ds + l_{\xi}^{\mu}(e^u)\\
&=e^{-u}\int_{e^u}^{\infty}ds\int_s^{\infty} r^{-1}\nu_{\xi}^{\rh}
(dr)+\int_{e^u}^{\infty}r^{-1}\nu_{\xi}^{\rh}(dr)\\
&=e^{-u}\int_{e^u}^{\infty} r^{-1}\nu_{\xi}^{\rh}(dr)\int_{e^u}^r ds+
\int_{e^u}^{\infty}r^{-1}\nu_{\xi}^{\rh}(dr)\\
&=\int_{e^u}^{\infty}r^{-1}\nu_{\xi}^{\rh}(dr)\left ( e^{-u}(r-e^u)+1\right )\\
&=e^{-u}\nu_{\xi}^{\rh}((e^u,\infty))
= h_{\xi}^{\rh}(u).
\end{align*}}
This completes the proof of (2).
\qed

\vskip 3mm
The next two propositions give us some properties of $\mu\in U_m(\rd)$
for $m\in\{0,1,2,...,\infty\}$.

\begin{prop}\label{p9}
Let $m\in\{0,1,\ldots\}$. Suppose that $\mu\in U_m(\R^d)$ with  
L\'evy measure $\nu\neq0$.
Let $(\ld(d\xi), \nu_{\xi}(dr))$ be a polar decomposition of $\nu$.
Let 
\begin{equation}\label{4.4a}
h_{\xi}(u)=e^{-u}\nu_{\xi}((e^u,\infty)),\qquad u\in\R .
\end{equation}
Then, for $\ld$-a.e.\ $\xi$, 
$h_{\xi}(u)$ is $m$ times differentiable on $\R$ and $(d/du)^m h_{\xi}(u)$ 
is absolutely continuous on $\R$.  Moreover,
\begin{equation}\label{4.5}
\left(-\frac{d}{du}\right)^j h_{\xi}(u)\ges  0\qquad\text{for all $u\in\R$ \,\, for 
$j=0,1,\ldots,m$}
\end{equation}
and 
\begin{equation}\label{4.6}
\left(-\frac{d}{du}\right)^{m+1} h_{\xi}(u)\ges 0\qquad\text{for }\,\,
\mrm{a.e.} \,\,u\in\R.
\end{equation}
\end{prop}

\n
{\it Proof}.  {\it Step 1}. The case $m=0$. We have $\mu\in U_0(\rd)=\mcal U(I(\rd))$.
Hence there is $\rh$ such that $\mu=\mcal U(\rh)$. 
We have $\nu^{\rh}\neq0$ for the L\'evy measure $\nu^{\rh}$ of $\rh$.
Let 
$(\ld^{\rh} (d\xi), \nu_{\xi}^{\rh}(dr))$ be a
polar decomposition of $\nu^{\rh}$.  Then Proposition \ref{p7} gives
a polar decomposition $(\ld^{\mu}(d\xi),\nu_{\xi}^{\mu}(dr))$ 
of the L\'evy measure  $\nu$ of $\mu$. 
On the other hand, $\nu$ has a polar decomposition 
$(\ld(d\xi), \nu_{\xi}(dr))$.
Hence it follows from Proposition \ref{p2.1} that there is a measurable function
$0<c(\xi)<\infty$ such that $\ld(d\xi)=c(\xi)\ld^{\mu}(d\xi)$ and, for 
$\ld$-a.e.\ $\xi$, $\nu_{\xi}(dr)=c(\xi)^{-1}\nu_{\xi}^{\mu}(dr)$. Hence
$h_{\xi}(u)=c(\xi)^{-1}h_{\xi}^{\mu}(u)$ for $\ld$-a.e.\ $\xi$. Thus,
Proposition \ref{p7} shows that, for $\ld$-a.e.\ $\xi$, $h_{\xi}(u)$ is 
absolutely continuous and
$(-d/du)h_{\xi}(u)=c(\xi)^{-1}h_{\xi}^{\rh}(u)\ges0$.

{\it Step 2}. Suppose that the statement is true for $m$.
Suppose that $\mu\in U_{m+1}(\rd)=\mcal U^{m+2}(I(\rd))$ with  
L\'evy measure $\nu\neq0$. Then there is $\rh\in\mcal U^{m+1}(I(\rd))$ 
such that $\mu=\mcal U(\rh)$.
The same argument as in Step 1 shows that, for $\ld$-a.e.\ $\xi$, $h_{\xi}(u)$ is 
absolutely continuous and
$(-d/du)h_{\xi}(u)=c(\xi)^{-1}h_{\xi}^{\rh}(u)$ for a.e.\ $u>0$.
Moreover $h_{\xi}(u)$ is differentiable and this equality holds
for all $u>0$, because 
$h_{\xi}^{\rh}(u)$ is continuous since $\rh\in U_m(\rd)\subset U_0(\rd)$.
Now, using the induction hypothesis, we see that $h_{\xi}^{\rh}(u)$
satisfies our assertion with $m$ replaced by $m+1$.
\qed

\vskip 3mm
\begin{prop}\label{p10}
Suppose that $\mu\in U_{\infty}(\R^d)$ with L\'evy measure $\nu\neq0$.
Let\linebreak $(\ld(d\xi), \nu_{\xi}(dr))$ be a polar decomposition of $\nu$.
Define $h_{\xi}(u)$ by \eqref{4.4a}.  Then, for $\ld$-a.e.\ $\xi$, $h_{\xi}(u)$ is 
completely monotone on $\R$.
\end{prop}

\n
{\it Proof}. This is clear from Proposition \ref{p9}. \qed

\medskip
Now we use Bernstein's theorem and the representation theorem for $L_{\infty}
(\R^d)$.

\begin{prop}\label{p11}
$U_{\infty}(\R^d)\subset L_{\infty}(\R^d)$.
\end{prop}

\n
{\it Proof}. 
Let $\mu\in U_{\infty}(\R^d)$.
If $\mu$ is Gaussian, then obviously $\mu\in L_{\infty}(\R^d)$.
Suppose that $\mu$ has L\'evy measure $\nu\neq0$. Choose the
polar decomposition $(\ld(d\xi), \nu_{\xi}(dr))$ of $\nu$ such that
\begin{equation*}
\int_0^{\infty}(r^2\land1)\nu_{\xi}(dr)=c=\int_{\R^d}(|x|^2\land1)\nu(dx),
\qquad \xi\in S
\end{equation*}
and $\ld(d\xi)$ is a probability measure.
Let $h_{\xi}(u)=e^{-u}\nu_{\xi}((e^u,\infty))$.
Then it follows from Proposition \ref{p10} that,  for $\ld$-a.e.\ $\xi$, 
$h_{\xi}(u)$ is completely monotone on $\R$.
For $a\in\R$, $h_{\xi}(a+u)$ is a completely monotone function of $u>0$.
Hence by Bernstein's theorem 
there is a unique measure $H_{\xi}^a$ on $[0,\infty)$ such that
\[
h_{\xi}(a+u)=\int_{[0,\infty)}e^{-uv} H_{\xi}^a(dv),\qquad u>0.
\]
If $a_1<a_2$, then
\[
h_{\xi}(a_2+u)=h_{\xi}(a_1+(a_2-a_1)+u)=\int_{[0,\infty)}e^{-(a_2-a_1+u)v}
 H_{\xi}^{a_1}(dv)
\]
for $u>0$ and, by the uniqueness, we have
\[
e^{-(a_2-a_1)v} H_{\xi}^{a_1}(dv)=H_{\xi}^{a_2}(dv).
\]
Hence $e^{av}H_{\xi}^{a}(dv)$ is independent of $a$. Write this measure
as $H_{\xi}(dv)$. Then
\[
h_{\xi}(a+u)=\int_{[0,\infty)}e^{-uv} e^{-av}H_{\xi}(dv)
\]
for all $a\in \R$ and $u>0$. It follows that
\begin{equation}\label{4.7}
h_{\xi}(u)=\int_{[0,\infty)} e^{-uv}H_{\xi}(dv),\qquad u\in\R.
\end{equation}
As in p.\,218 of [10] or p.\,17 of [8],
we can prove that, for any $B\in\mcal B([0,\infty))$, $H_{\xi}(B)$ is
measurable in $\xi$. The identity \eqref{4.7} can be written as
\begin{equation}\label{4.8}
\nu_{\xi}((r,\infty))=\int_{[0,\infty)} r^{1-v}H_{\xi}(dv),\qquad r>0.
\end{equation}
Since the left-hand side tends to $0$ as $r\to\infty$, we obtain 
$H_{\xi}(\,[0,1]\,)=0$. Now we get
\[
-\frac{d}{dr}(\nu_{\xi}((r,\infty)))=\int_{(1,\infty)} (v-1)r^{-v}H_{\xi}(dv),
\qquad \text{a.e. }r>0.
\]
Since
\[
\infty>\int_0^1 r^2\nu_{\xi}(dr)=\int_0^1 r^2dr\int_{(1,\infty)}(v-1)r^{-v}H_{\xi}(dv)
=\int_{(1,\infty)}(v-1)H_{\xi}(dv)\int_0^1 r^{2-v}dr,
\]
we obtain $H_{\xi}(\,[3,\infty)\,)=0$. Define
\[
\Gm_{\xi}(E)=\int_{(1,3)} 1_E (v-1)\,(v-1)H_{\xi}(dv),\qquad E\in\mcal B([0,\infty)).
\]
Then $\Gm_{\xi}$ is concentrated on $(0,2)$ and
\[
\int_{(0,2)} f(\al)\Gm_{\xi}(d\al)=\int_{(1,3)}f(v-1)\,(v-1)H_{\xi}(dv),\qquad 
\text{for all measurable $f\ges0$}.
\]
Now,
{\allowdisplaybreaks
\begin{gather*}
\frac{d}{dr}(\nu_{\xi}((r,\infty)))=\int_{(1,3)}(v-1)r^{-(v-1)-1}H_{\xi}(dv)
=\int_{(0,2)}r^{-\al-1}\Gm_{\xi}(d\al),\\
\int_{(0,1]} r^2\nu_{\xi}(dr)=\int_{(0,2)}\frac{\Gm_{\xi}(d\al)}{2-\al},\\
\int_{(1,\infty)} \nu_{\xi}(dr)=\int_{(0,2)}\frac{\Gm_{\xi}(d\al)}{\al}.
\end{gather*}
Hence}
\[
\int_{(0,2)}\left(\frac{1}{\al}+\frac{1}{2-\al}\right)\Gm_{\xi}(d\al)=c.
\]
We can find a finite measure $\Gm$ on $(0,2)$ and probability measures $\ld_{\al}$
on $S$ such that $\ld_{\al}(D)$ is measurable in $\al$ for any $D\in\mcal B(S)$ and
$\Gm(d\al)\ld_{\al}(d\xi)=\ld(d\xi)\Gm_{\xi}(d\al)$. Thus,
\[
\int_{(0,2)}\left(\frac{1}{\al}+\frac{1}{2-\al}\right)\Gm(d\al)=c
\]
and,
for any $B\in\mcal B(\R^d\setminus\{0\})$, 
\begin{align*}
\nu(B)&=\int_S \ld(d\xi)\int_0^{\infty} 1_B(r\xi)\nu_{\xi}(dr)
=\int_S \ld(d\xi)\int_0^{\infty} 1_B(r\xi)dr\int_{(0,2)}r^{-\al-1}
\Gm_{\xi}(d\al)\\
&=\int_{(0,2)}\Gm(d\al)\int_S \ld_{\al}(d\xi)\int_0^{\infty} 1_B(r\xi)
r^{-\al-1}dr.
\end{align*}
This is exactly the form of the L\'evy measure 
in Theorem 22 of [8] (originally Theorem 3.4
of [10]).  This shows that $\mu\in L_{\infty}(\R^d)$.
The proof is completed.
\qed

\vskip 3mm
Finally we have the following.

\begin{prop}
$U_{\infty}(\R^d)= L_{\infty}(\R^d)$.
\end{prop}

\n
{\it Proof}. It remains to prove that $U_{\infty}(\R^d)\supset L_{\infty}(\R^d)$, 
but this is concluded from that $U_{m}(\R^d)\supset L_{m}(\R^d)$ for each $m\ge 1$,
which is shown in [6].
\qed

\vskip 10mm
\n
{\bf Acknowledgement}\\
The authors wish to thank Z.J. Jurek for a discussion on his paper [6].


\end{document}